\documentclass[12pt]{amsart}

 \usepackage{amsfonts,graphics,amsmath,amsthm,amsfonts,amscd,amssymb,amsmath,latexsym,multicol,
 mathrsfs}
\usepackage{epsfig,url}
\usepackage{flafter}
\usepackage{fancyhdr}
\usepackage{hyperref}
\hypersetup{colorlinks=true, linkcolor=black}

\makeatletter

\def\jobis#1{FF\fi
  \def\predicate{#1}%
  \edef\predicate{\expandafter\strip@prefix\meaning\predicate}%
  \edef\job{\jobname}%
  \ifx\job\predicate
}

\makeatother

\if\jobis{proposal}%
\else
\fi

 \usepackage[matrix, arrow]{xy}

\DeclareMathOperator{\Supp}{Supp}

\DeclareMathOperator{\Fix}{Fix}

 \numberwithin{equation}{subsection}
 \numberwithin{footnote}{subsection}

 \newtheorem{cor}[subsection]{Corollary}
 \newtheorem{lem}[subsection]{Lemma}
 
 \newtheorem{thm}[subsection]{Theorem}
 \newtheorem{conj}[subsection]{Conjecture}

\newtheorem{claim}[subsection]{Claim}
 \newtheorem{quest}[subsection]{Question}

{
    \newtheoremstyle{upright}%
        {8pt plus2pt minus4pt}%
        {8pt plus2pt minus4pt}%
        {\upshape}%
        {}%
        {\bfseries\scshape}%
        {}%
        {1em}%
        {}%
\theoremstyle{upright}

}

 \newcommand{\N}{\mathbb N}
 \newcommand{\PP}{\mathbb P}
 
 \newcommand{\Q}{\mathbb Q}
 \newcommand{\R}{\mathbb R}
 
 \newcommand{\bir}{\dashrightarrow}
 \newcommand{\rddown}[1]{\left\lfloor{#1}\right\rfloor} 

\title{On log minimal models and Zariski decompositions II}
\author{Caucher Birkar and Zhengyu Hu}
\date{\today}
\begin{document}
\maketitle

\begin{abstract}
We continue our study of the relation between log minimal models and various types of Zariski decompositions.
Let $(X,B)$ be a projective log canonical pair. We will show that $(X,B)$ has a log minimal model if
either $K_X+B$ birationally has a
Nakayama-Zariski decomposition with nef positive part, or that $K_X+B$ is big and birationally it has
a Fujita or CKM Zariski decomposition. Along the way we introduce polarized pairs $(X,B+P)$
where $(X,B)$ is a usual projective pair and $P$ is nef, and study the birational geometry
of such pairs.
\end{abstract}



\section{Introduction}

We will work over an algebraically closed field $k$ of characteristic zero. In this paper we continue
our study of the relation between log minimal models and various types of Zariski decompositions.
Let $(X,B)$ be a projective log canonical (lc) pair.
The main result of [\ref{B-WZD}] states that if $K_X+B$ birationally has a weak Zarsiki decomposition, then
the pair has a log minimal model assuming that the log minimal model program (LMMP) holds in lower dimension,
in particular, this assumes termination of log flips in lower dimension.
In this paper, we show that if we take a stronger form of Zariski decomposition, then we can construct a
log minimal model without any extra assumptions in lower dimension (for simplicity we
state our results in the absolute projective case but everything can be formulated and proved
in a similar way in the relative setting). More precisely:

\begin{thm}\label{t-main-1}
Let $(X,B)$ be a projective lc pair with $K_X+B$ pseudo-effective.
If $K_X+B$ birationally has a Nakayama-Zariski decomposition with nef positive part, then
$(X,B)$ has a log minimal model.
\end{thm}

The assumption of the theorem is that on some resolution $f\colon W\to X$ we have the Nakayama-Zariski decomposition
$f^*(K_X+B)=P_\sigma+N_\sigma$ with $P_\sigma$ nef (in general $P_\sigma$ is only pseudo-effective).
For more details on the terminology and the proof see Section \ref{s-NZD}.

Our next result concerns the Fujita and CKM Zariski decompositions:

\begin{thm}\label{t-main-2}
Let $(X,B)$ be a projective lc pair such that $K_X+B$ is big. Then, the following are
equivalent:

(1) $K_X+B$ birationally  has a Fujita-Zariski decomposition;

(2) $K_X+B$ birationally  has a CKM-Zariski decomposition;

(3) $(X,B)$ has a log minimal model;
\end{thm}

 The theorem is similar to a result of Kawamata [\ref{Kawamata-ZD}] for klt pairs. 
For the proof see Section \ref{s-FCKMZD}.

In order to prove the above theorems we consider a class of pairs beyond the traditional lc pairs.
We look at pairs $(X,B+P)$ in which $(X,B)$ is a usual projective pair and $P$ is a nef $\R$-divisor: we call the pair
a polarized pair. Besides this paper, polarized pairs appear in other contexts, eg the base point free theorem, 
canonical bundle formulas [\ref{Kawamata-subadjuntion}][\ref{Fujino-Mori}],    
proof of [\ref{BC}, Theorem 1.2], [\ref{Cacciola}], etc.
It is interesting to see if one can extend the birational geometry of usual pairs to the setting of
polarized pairs. For example, the cone and contraction theorems and the existence of log flips hold for polarized pairs.
One can ask whether one can run an LMMP on $K_X+B+P$ which terminates.
One can ask many other questions some of which are treated in Section \ref{s-p-pairs}.

\begin{conj}
Let $(X,B+P)$ be a $\Q$-factorial dlt polarized pair.
Then, we can run an LMMP on $K_X+B+P$ which terminates.
\end{conj}

By $(X,B+P)$ being $\Q$-factorial dlt we mean that $(X,B)$ is $\Q$-factorial dlt.
Our first result concerning the conjecture is the following:

\begin{thm}\label{t-main-3}
Let $(X,B+P)$ be a $\Q$-factorial dlt polarized pair.
Assume that either $P$ or $N=K_X+B-P$ is a $\Q$-divisor, or $N \geq 0$ with $\Supp N \subseteq \Supp B$.
 Moreover, assume that every LMMP on $K_X+B$ terminates.
Then, for any $\alpha\ge 0$ we can run an LMMP on $K_X+B+\alpha P$ which terminates.
Moreover, if $\alpha\gg 0$, then the LMMP is $P$-trivial.
\end{thm}

Here $P$-trivial means that $P$ is numerically trivial on every extremal ray in the
process, in particular, this means that the nefness of $P$ is preserved in this case.

\begin{cor}\label{c-dim-3}
Let $(X,B+P)$ be a $\Q$-factorial dlt polarized pair of dimension $\le 3$.
Assume that either $P$ or $N=K_X+B-P$ is a $\Q$-divisor, or $N \geq 0$ with $\Supp N \subseteq \Supp B$.
Then, for any $\alpha\ge 0$ we can run an LMMP on $K_X+B+\alpha P$ which terminates.
Moreover, if $\alpha\gg 0$, then the LMMP is $P$-trivial.
\end{cor}

Most probably the $\Q$-divisor condition in these results is unnecessary.
In Section \ref{s-p-pairs} we also touch upon limiting pairs which are more general than polarized pairs.
The following theorem shows that Conjectures 2 and 2b of [\ref{Cacciola}] hold. 

\begin{thm}\label{t-main-4}
Let $(X,B+P)$ be a lc polarized pair where $B,P$ are $\Q$-divisors, and assume that 
$K_X+B+P$ is big. Further assume that the augmented base locus 
${\bf{B}}_+(K_X+B+P)$ does not contain any lc centre of $(X,B)$. If $K_X+B+P$ birationally 
has a CKM-Zariski decomposition, then   
the section ring $R(K_X+B+P)$ is a finitely generated $k$-algebra.   
\end{thm}

The proof is given in Section \ref{s-FCKMZD}. If $(X,B+P)$ is $\Q$-factorial dlt, the 
proof is trivial so the difficulty has to do with lc singularities. 

In the final section of this paper, 
Section \ref{s-WZD}, we outline a strategy to show that: if every pseudo-effective lc pair of 
dimension $\le d$ has a weak Zariski decomposition, then every such pair has a log minimal model. This
is an attempt to remove the termination assumption in  the main theorem of [\ref{B-WZD}].\\

\textbf{Acknowledgements.} The first author was supported by a Leverhulme grant. The second author was supported by the EPSRC grant EP/I004130/1. The authors would like to
thank Yoshinori Gongyo for useful discussions.

\section{Preliminaries}

Let $k$ be an algebraically closed field of characteristic zero fixed throughout the paper. 
All the varieties will be over $k$ unless stated otherwise. For an $\R$-divisor $M=\sum m_iM_i$ on some  
variety, we define $||M||=\max\{m_i\}$. \\

\textbf{Pairs.} 
A \emph{pair} $(X,B)$ consists of a normal quasi-projective variety $X$ and an $\R$-divisor $B$ on $X$ with
coefficients in $[0,1]$ such that $K_X+B$ is $\mathbb{R}$-Cartier. 
For a prime divisor $D$ on some birational model of $X$ with a
nonempty centre on $X$, $a(D,X,B)$
denotes the log discrepancy. For definitions and standard results on singularities of pairs 
we refer to [\ref{Kollar-Mori}].\\

\textbf{Log minimal models and Mori fibre spaces.}
A projective pair $(Y,B_Y)$ is a \emph{log birational model} of a projective pair $(X,B)$ if we are given a birational map
$\phi\colon X\bir Y$ and $B_Y=B^\sim+E$ where $B^\sim$ is the birational transform of $B$ and 
$E$ is the reduced exceptional divisor of $\phi^{-1}$, that is, $E=\sum E_j$ where $E_j$ are the
exceptional/$X$ prime divisors on $Y$. 
A log birational model $(Y,B_Y)$ is a  \emph{weak lc model} of $(X,B)$ if

$\bullet$ $K_Y+B_Y$ is nef, and

$\bullet$ for any prime divisor $D$ on $X$ which is exceptional/$Y$, we have
$$
a(D,X,B)\le a(D,Y,B_Y)
$$

 A weak lc model $(Y,B_Y)$ is a \emph{log minimal model} of $(X,B)$ if 

$\bullet$ $(Y,B_Y)$ is $\Q$-factorial dlt,

$\bullet$ the above inequality on log discrepancies is strict.\\

On the other hand, a log birational model $(Y,B_Y)$  is called a \emph{Mori fibre space} of $(X,B)$ if 

$\bullet$ $(Y,B_Y)$ is $\Q$-factorial dlt, 

$\bullet$ there is a $K_Y+B_Y$-negative extremal contraction $Y\to T$ 
with $\dim Y>\dim T$, and 

$\bullet$ for any prime divisor $D$ (on birational models of $X$) we have 
$$
a(D,X,B)\le a(D,Y,B_Y)
$$
 and strict inequality holds if $D$ is 
on $X$ and contracted$/Y$.\\

Note that our definitions of log minimal models and Mori fibre spaces are slightly different 
from the traditional definitions in that we allow $\phi^{-1}$ to contract certain divisors.\\

\textbf{Rational decomposition of certain divisors.} 
Let $(X,B)$ be a $\Q$-factorial lc pair such that $K_X+B=P+N$ where $P$ is an $\R$-divisor and $N$ is a $\Q$-Cartier. 
For each $\delta>0$, we will show that there is a decomposition $P=\sum r_iP_i$ satisfying the following:

$\bullet$ $P_i$ are $\Q$-Cartier, 

$\bullet$ $r_i$ are positive and linearly independent over $\Q$, 

$\bullet$ $||P-P_i||<\delta$ and $\Supp (P-P_i)\subseteq \Supp B$, and 

$\bullet$ $(X,B_i)$ is lc where $B_i:=P_i+N-K_X$.
 
Let $V$ be the $\R$-vector space generated by the components of $K_X,B,P,N$. In what follows a \emph{rational 
vector space} (inside $V$) means an $\R$-vector space generated by a collection of $\Q$-divisors, 
and a \emph{rational affine space} (inside $V$) means an affine space spanned by a collection of $\Q$-divisors. 

Let $W\subseteq V$ be the 
smallest rational affine space containing $P$, and let $U$ be the vector space generated by the elements of $W$. 
If $Q\in W$, then $P-Q$ is supported on the irrational part of $P$, and since $N$ is a $\Q$-divisor 
the irrational part of $P$ is supported in the irrational part of $B$ hence $\Supp(P-Q)\subseteq \Supp B$.

First assume that $W\neq U$. Then, we can write $P=\sum_1^n r_iP_i$ where $P_i\in W$ are $\Q$-divisors, $r_i>0$, 
and $n=\dim W+1=\dim U$. The $r_i$ are 
linearly independent over $\Q$: otherwise we have $\sum a_ir_i=0$ for certain $a_i\in \Q$ (say $a_n=1$) 
hence $P=\sum_1^{n-1}r_iP_i'$ where $P_i'=P_i-a_iP_n$, and this implies that 
$P$ belongs to a rational vector space strictly smaller than $U$, a contradiction.
Now assume that $W=U$. In this case, we can choose $\Q$-divisors $P_i$ which form a basis of $W$ 
and such that $P$ belongs to the convex cone generated by the $P_i$, that is, $P=\sum_1^n r_iP_i$ 
for certain $r_i>0$. The $r_i$ are linearly independent over $\Q$ otherwise $P$ would belong to a 
rational vector space strictly smaller than $W$, a contradiction. In both cases $W\neq U$ and $W=U$, 
it is obvious that we can choose the $P_i$ so that $||P-P_i||<\delta$.

It remains to show that the $P_i$ can be chosen so that $(X,B_i)$ is lc where $B_i:=P_i+N-K_X$.  
Put $W'=W+N-K_X$ which is a rational affine space containing $B$. By Shokurov [\ref{Sh-log-models}], the set 
$$
\mathcal{L}=\{\Delta \in W' \mid (X,\Delta) ~~\mbox{is lc}\}
$$
is a rational polytope. By our choice of $W$, $B$ belongs to the interior of $\mathcal{L}$: 
otherwise $B$ would belong to some proper face of $\mathcal{L}$ hence there would be a rational affine space $T'\subsetneq W'$ containing $B$ which in turn implies $P$ belongs to the rational affine space $T:=T'+K_X-N\subsetneq W$, a contradiction. The same reasoning shows that $\dim \mathcal{L}=\dim W$. 
Therefore, if we choose $P_i$ so that $||P-P_i||=||B-B_i||$ is sufficiently small, 
then each $B_i$ belongs to the interior of $\mathcal{L}$ hence in particular each $(X,B_i)$ is lc.

\vspace{0.3cm}
\section{Polarized pairs}\label{s-p-pairs}

A \emph{polarized pair} is of the form $(X,B+P)$ where $(X,B)$ is a projective pair in the usual sense
and $P$ is a nef $\R$-divisor on $X$.
We will call $K_X+B+P$ a \emph{polarized log divisor}.
We say that a polarized pair $(X,B+P)$ is lc (dlt, etc) if $(X,B)$ is lc (resp. dlt, etc).
In this section, we will look at various questions concerning polarized pairs.
Essentially we would like to know how much of the birational geometry of usual pairs
can be extended to polarized pairs.\\

\textbf{Cone and contraction, flips, LMMP}. For simplicity assume that $(X,B)$ is dlt. We can run the LMMP
on $K_X+B+P$ as follows. Let $R$ be a $K_X+B+P$-negative extremal ray.
Pick an ample $\R$-divisor $A$ such that $(K_X+B+P+A)\cdot R<0$. We have
$K_X+\Delta\sim_\R K_X+B+P+A$ for some boundary $\Delta$ so that $(X,\Delta)$ is dlt.
Then, by the cone theorem $R$ can be contracted. This shows that the cone and contraction theorems
hold for polarized pairs. Moreover, the log flip exists if $R$ defines
a flipping contraction. We can continue the process, that is, if we have already obtained a
model $X'$ and if we have a $K_{X'}+B'+P'$-negative extremal ray $R'$, then perhaps after replacing
$A$ we can  make sure that $(K_{X'}+B'+P'+A')\cdot R'<0$ and that $(X,\Delta')$ is dlt where $'$ denotes birational
transform. So, $R'$ can be
contracted, etc.  The process gives an LMMP on $K_X+B+P$.

\begin{quest}
Does the LMMP just defined terminates?
\end{quest}

Under some mild assumptions, we will show that at least some
LMMP on $K_X+B+P$ terminates if every LMMP on $K_X+B$ terminates. In particular, we can
apply this in dimension $\le 3$ since the latter termination is known.

\begin{thm}\label{t-p-trivial-LMMP}
Let $(X,B+P)$ be a $\Q$-factorial dlt polarized pair.
Assume that either $P$ or $N=K_X+B-P$ is a $\Q$-divisor, or $N \geq 0$ with $\Supp N \subseteq \Supp B$.
Then, for any $\alpha\gg 0$ any LMMP on $K_X+B+\alpha P$ is $P$-trivial.
\end{thm}
\begin{proof}
Here $P$-trivial means that $P$ is numerically trivial on every extremal ray in the
process. First assume that $N=K_X+B-P$ is a $\Q$-divisor.
By the arguments in Section 2, we can write $P=\sum r_iP_i$ where  $P_i$ are 
$\Q$-divisors, $r_i>0$ are linearly independent over $\Q$, and each $(X,B_i)$ is lc.

Pick $\alpha>0$ and let $R$ be an extremal ray such that $(K_X+B+\alpha P)\cdot R<0$.
Then, $(K_X+B)\cdot R<0$ and $N\cdot R<0$.
Let $\Gamma$
be an extremal curve generating $R$ ($\Gamma$ is a curve generating $R$ having minimal degree with
respect to some ample divisor, see [\ref{Sh-log-models}][\ref{Sh-ordered}] or [\ref{B-II}, Section 3] for more details). 
Then, as a consequence of the boundedness of the length of extremal rays [\ref{Kawamata-ray}] 
we have $(K_X+B)\cdot \Gamma\ge -2\dim X$ which implies that
$$
\alpha P\cdot \Gamma <-(K_X+B)\cdot \Gamma\le 2\dim X
$$
On the other hand, if $P_i\cdot \Gamma<0$, then $(K_X+B_i)\cdot \Gamma<0$ and
$$
P_i\cdot \Gamma=(K_X+B_i)\cdot \Gamma-N\cdot \Gamma>(K_X+B_i)\cdot \Gamma\ge -2\dim X
$$
Thus, $P_i\cdot \Gamma\ge -2\dim X$ for every $i$.
Now pick a positive integer $m$ such that $mP_i$ is Cartier for every $i$. Then,
$$
P\cdot \Gamma=\sum r_iP_i\cdot \Gamma=\sum r_i\frac{n_i}{m}
$$
where $n_i$ are integers and $n_i\ge -2m\dim X$. This implies that $P\cdot \Gamma=0$ or 
$P\cdot \Gamma>\mu$ where $\mu>0$
only depends on $m$ and $r_i$. In particular, if $\alpha\gg 0$ (depending only on $m,r_i$), then from
$\alpha P\cdot \Gamma<2\dim X$ we deduce that $P\cdot \Gamma=0$.

From now on we fix $\alpha\gg 0$. By construction, $\sum r_iP_i\cdot \Gamma=P\cdot \Gamma=0$.
Since the $r_i$ are linearly independent over $\Q$, we have $P_i\cdot \Gamma=0$ for
every $i$. In particular, $(K_X+B_i)\cdot \Gamma=N\cdot \Gamma<0$ for each $i$. 
If $R$ defines a Mori fibre space, then the LMMP stops. Otherwise, let $X\bir X'$ be the
divisorial contraction or the flip associated to $R$. Then, $K_{X'}+B'=P'+N'$,
$P'$ is nef, $mP_i$ is Cartier, etc,
so we can apply the above arguments on $X'$
and in this way obtain an LMMP on $K_X+B+\alpha P$ which is $P$-trivial in every step.

Now we treat the case when $P$ is a $\Q$-divisor.
Pick a sufficiently large number $\alpha$. Let $m>0$ be an integer such that $mP$ is Cartier.
Let $R$ be an extremal ray such that $(K_X+B+\alpha P)\cdot R<0$ and let $\Gamma$ be an
extremal curve generating $R$. Then, $\alpha P\cdot \Gamma=\alpha \frac{n}{m}$ for some
integer $n\ge 0$ and from
$$
\alpha P\cdot \Gamma <-(K_X+B)\cdot \Gamma\le 2\dim X
$$
we deduce that $P\cdot \Gamma=0$ since $\alpha$ is sufficiently large. The rest of the argument
goes as before.

Now we come to the third case, that is, when $N\ge 0$ and $\Supp N\subseteq \Supp B$. 
First pick a sufficiently small $\epsilon>0$ and let $N'$ be a $\Q$-divisor such that we have 
$(1-\epsilon)N\le N'\le N$. Put $B'=P+N'-K_X$. Then, $(X,B')$ is dlt. By the constructions in Section 2, 
we can write $P=\sum r_iP_i$ where $P_i$ are $\Q$-divisors, $r_i>0$ are linearly independent 
over $\Q$, $||P-P_i||$ are sufficiently small with $P-P_i$ supported in $\Supp B'=\Supp B$, 
and $(X,B_i')$ are lc where $B_i':=P_i+N'-K_X$.  
Let $B_i:=P_i+N-K_X$. Although $(X,B_i)$ may not be lc but $(X,B_i-\epsilon N)$ is lc 
because $B_i-\epsilon N\le B_i'$. 
Now let $R$ be an extremal ray such that 
$(K_X+B+\alpha P)\cdot R<0$ and let $\Gamma$ be an extremal curve generating $R$. 
Then, 
$$
P_i\cdot \Gamma=(K_X+B_i-\epsilon N) \cdot \Gamma - (1-\epsilon) N\cdot \Gamma\ge -2\dim X
$$ 
The rest of the argument is similar to the case when we assumed $N$ to be a $\Q$-divisor. \\
\end{proof}

\begin{thm}\label{t-l-pairs-termination}
Let $(X,B+P)$ be a $\Q$-factorial dlt polarized pair.
Assume that either $P$ or $N=K_X+B-P$ is a $\Q$-divisor, or $N \geq 0$ with $\Supp N \subseteq \Supp B$.
 Moreover, assume that every LMMP on $K_X+B$ terminates.
Then, for any $\alpha\ge 0$ we can run an LMMP on $K_X+B+\alpha P$ which terminates.
\end{thm}
\begin{proof}
Pick a  sufficiently large number $\beta>\alpha$. By Theorem \ref{t-p-trivial-LMMP}, we can run an
LMMP on $K_X+B+\beta P$ which is $P$-trivial. The LMMP is an LMMP on both
$K_X+B+\alpha P$ and $K_X+B$ so it terminates by assumptions.
Thus, we reach a model $Y$ on which $K_Y+B_Y+\beta P_Y$ is nef or there is a
$K_Y+B_Y+\beta P_Y$-negative Mori fibre space structure. In the latter case we are done since the
fibre space structure is also $K_Y+B_Y+\alpha P_Y$-negative.
So, assume the former case and by replacing $X$ with $Y$
we may assume that $K_X+B+\beta P$ is nef.

We will run an LMMP on $K_X+B$ with scaling of $\beta P$. More precisely, we replace
$\beta$ with the number
$$
\min\{t\ge \alpha \mid K_X+B+tP ~~\mbox{is nef}\}
$$
If $\beta=\alpha$, we are done. So, we can assume that  $\beta>\alpha$.

\begin{claim}
There is an extremal ray $R$ such that
$(K_X+B)\cdot R<0$ but $(K_X+B+\beta P)\cdot R=0$.
\end{claim}
\emph{Proof of the Claim.} 
If the claim is not true, then
there exist a strictly increasing sequence $t_i$ of positive numbers approaching $\beta$, and
extremal rays $R_i$ such that $(K_X+B)\cdot R_i<0$ but $(K_X+B+t_i P)\cdot R_i=0$.
For each $i$, pick an extremal curve $\Gamma_i$ generating $R_i$ (cf. [\ref{B-II}, Section 3]).
Then, the set of the intersection numbers
$\{(K_X+B)\cdot \Gamma_i\}_i$ is finite (cf. [\ref{B-II}, Remark 3.1]). So, the intersection
number
$$
P\cdot \Gamma_i=\frac{-(K_X+B)\cdot \Gamma_i}{t_i}
$$
is bounded from above and below. 

If $P$ is a $\Q$-divisor, then 
there are only finitely many possibilities for the numbers $P\cdot \Gamma_i$.
Therefore,  in this case there are only finitely many possibilities for the numbers $t_i$
and we should have $t_i=\beta$ for $i\gg 0$.
On the other hand, if $N$ is a $\Q$-divisor, then
from
$$
P\cdot \Gamma_i=(K_X+B)\cdot \Gamma_i-N\cdot \Gamma_i
$$
and the boundedness of $P\cdot \Gamma_i$ we deduce that $N\cdot \Gamma_i$ is bounded hence 
there are only finitely many
possibilities for the numbers $N\cdot \Gamma_i$. This in turn implies that
there are only finitely many
possibilities for the numbers $P\cdot \Gamma_i$ hence $t_i=\beta$ for $i\gg 0$.
Now assume that $N \geq 0$ with $\Supp N \subseteq \Supp B$. We will proceed similar to the 
proof of [\ref{B-WZD}, Lemma 3.1]. First we replace $N$ with $(1-\epsilon)N$ and replace 
$B$ with $B-\epsilon N$ for a sufficiently small $\epsilon$. Then, there is a $\Q$-divisor 
$N'$ with the same support as $N$ and such that $||N-N'||$ is sufficiently small, $(X,B')$ is lc where $B'=P+N'-K_X$, and 
$(K_X+B')\cdot \Gamma_i<0$ for every $i$ perhaps after replacing the sequence $t_i$ with an 
infinite subsequence. As above 
 the set of the intersection numbers
$\{(K_X+B')\cdot \Gamma_i\}_i$ is finite. Moreover, from 
$$
P\cdot \Gamma_i=(K_X+B')\cdot \Gamma_i-N'\cdot \Gamma_i
$$
and the boundedness of $P\cdot \Gamma_i$ we deduce that $N'\cdot \Gamma_i$ is also bounded 
hence there are only finitely many possibilities for $N'\cdot \Gamma_i$ 
which in turn implies that there are only finitely many possibilities for $P\cdot \Gamma_i$ 
hence from $(K_X+B+t_i P)\cdot \Gamma_i=0$ we get $t_i=\beta$ for $i\gg 0$. 
This proves the claim. 

We continue the proof of the theorem. If $R$ in the claim defines a Mori fibre space structure, 
we stop. Otherwise, let $X\bir X'$ be the divisorial contraction or the flip associated
to $R$. Continuing
the process on $X'$ and so on produces an LMMP on $K_X+B$ with scaling of $\beta P$.
Of course we may lose the nefness of $P$ but we do not need it.
The above LMMP is obviously an LMMP on $K_X+B$. So, by assumptions, it terminates.
Therefore, along the way we get a model $Y$ on which $K_Y+B_Y+\alpha P_Y$ is nef or there is a
$K_Y+B_Y+\alpha P_Y$-negative Mori fibre space structure and we
are done.\\
\end{proof}

\begin{proof}(of Theorem \ref{t-main-3})
This follows from Theorems \ref{t-p-trivial-LMMP} and \ref{t-l-pairs-termination}.
\end{proof}

\begin{proof}(of Corollary \ref{c-dim-3})
Since the LMMP holds in dimension $\le 3$, the claim follows from Theorem \ref{t-main-3}.
\end{proof}

\textbf{Abundance}. Although the LMMP is expected to hold for polarized pairs by the
above results but abundance does not hold. Indeed if $X$ is an
elliptic curve, $B=0$ and $P$ is a non-torsion numerically trivial divisor, then
$K_X+B+P$ is nef but not semi-ample. A more subtle counter-example is the following.
There is a smooth projective surface $X$ which is a ruled surface over some elliptic curve
with a section $S$ such that $K_X+2S\sim 0$ and $S$ is nef but not semi-ample
(cf. [\ref{Sh-complements}, Example 1.1]).
Now put $B=S$ and $P=2S$. Then, $K_X+B+P\sim S$ which is nef but not semi-ample.\\

\textbf{Finite generation}. Similar to abundance, finite generation also fails in general 
although it holds in certain interesting cases, eg Theorem \ref{t-main-4}.
Let $C$ be an elliptic curve and $Q$ a non-torsion numerically trivial divisor. Let
$X=\PP(\mathcal{O}_C\oplus \mathcal{O}_C(1))$ and $g\colon X\to C$ the corresponding morphism.
There is a birational morphism $f\colon X\to Z$ which contracts a curve $S$: the morphism is
induced by the globally generated invertible sheaf $\mathcal{O}_X(1)$, and the morphism $S\to C$ is
an isomorphism. There is also a curve $T$ disjoint from $S$ given by a general section of $\mathcal{O}_X(1)$
such that again the morphism $T\to C$ is an isomorphism.
It is easy to see that $K_X+S+T\sim 0$. Let $B=S+T$ and $P=g^*Q+A$ where $A$ is the divisor corresponding
to $\mathcal{O}_X(1)$.
Then, $K_X+B+P$ is nef and big but not semi-ample hence its algebra is not finitely generated.\\

\textbf{Non-vanishing.} As mentioned earlier if we take  $X$ to be an
elliptic curve, $B=0$ and $P$ a non-torsion numerically trivial divisor, then
$K_X+B+P$ is nef but not semi-ample. Moreover, the Kodaria dimension of $K_X+B+P$ is $-\infty$.
However, up to numerical equivalence $K_X+B+P$ has non-negative Kodaira dimension.
One can then ask:

\begin{quest}
Let $(X,B+P)$ be a lc polarized pair with $K_X+B+P$ pseudo-effective.
Is there $M\ge 0$ such that $K_X+B+P\equiv M$?
\end{quest}

Most probably the answer is no. However, there are interesting cases in which the answer to the
question is yes: for example assume that $(X,B+P)$ is a klt polarized pair with $B,P$ being $\Q$-Cartier, 
and let $X\to Z$ be a morphism where $Z$ is an abelian variety
and $K_X+B+P$ is big$/Z$; then $K_X+B+P\equiv M$ for some $M\ge 0$.
This can be proved as in [\ref{CCP}, Theorem 3.1] using Fourier-Mukai transforms (a complete proof
is given in [\ref{BCT}]).\\

\textbf{Limiting pairs.} A  \emph{limiting pair} is of the form $(X,\Delta)$ where $X$
is projective and $K_X+\Delta$ is $\R$-Cartier, and
 there exists a sequence of boundaries $\Delta_i$ such that $(X,\Delta_i)$ are pairs in the usual sense
 and $K_X+\Delta=\lim (K_X+\Delta_i)$ in $N^1(X)$. We say $(X,\Delta)$ is lc if we can choose
 $(X,\Delta_i)$ to be lc.

Every lc polarized pair $(X,\Delta=B+P)$ is a lc limiting pair:
we can pick ample $\Q$-divisors $A_i$ with $\lim A_i=0$ and choose appropriate $\Delta_i\sim_\R B+P+A_i$
so that  $(X,\Delta_i)$ are lc pairs in the usual sense
 and $K_X+\Delta=\lim (K_X+\Delta_i)$.

If $(X,\Delta)$ is a limiting pair and if
$X\bir Y$ is a partial LMMP on $K_X+\Delta$, then $(Y,\Delta_Y)$ is also a limiting pair.
The point is that $X\bir Y$ is also a partial LMMP on $K_X+\Delta_i$ for every $i\gg 0$
so the pairs $(Y,\Delta_{i,Y})$ are all lc in the usual sense when $i\gg 0$.

The cone and contraction theorems and the existence of log flips hold for every lc limiting pair.
One then wonders if a reasonable birational theory can be developed for such pairs.

\vspace{0.3cm}
\section{Nakayama-Zariski decomposition with nef positive part}\label{s-NZD}

Nakayama [\ref{Nakayama}] defined a decomposition $D=P_\sigma(D)+N_\sigma(D)$ for any pseudo-effective
$\R$-divisor $D$ on a smooth projective variety. We refer to this as the Nakayama-Zariski decomposition.
We call $P_\sigma$ the positive part and $N_\sigma$ the negative part. We can
extend it to the singular case as follows.
Let $X$ be a normal projective variety and $D$ a pseudo-effective $\R$-Cartier divisor on $X$. We define $P_\sigma(D)$
by taking a resolution $f\colon W\to X$ and letting $P_\sigma(D):=f_*P_\sigma(f^*D)$.
The following lemma shows in particular that this is independent of the choice of the resolution.

\begin{lem}\label{l-NZD}
With $X,D, W$ as above we have:

(1) $P_\sigma(D)$ is independent of the choice of the resolution $f$;

(2) For any $\R$-divisor $E\ge 0$ on $W$ which is exceptional$/X$, we have
$P_\sigma(f^*D+E)=P_\sigma(f^*D)$;

(3) If $0\le M$ is $\R$-Cartier with $\Supp M\subseteq \Supp N_\sigma(D)$, then $P_\sigma(D+M)=P_\sigma(D)$;

(4)  If $0\le M \le N_\sigma(D)$ is $\R$-Cartier, then $P_\sigma(D-M)=P_\sigma(D)$;

(5) If $\phi\colon X\bir Y$ is a partial $D$-MMP, then
$\phi_*P_\sigma(D)=P_\sigma(\phi_*D)$;

(6) If $A$  is $\R$-Cartier and $D+\epsilon A$ is a movable $\R$-divisor for each sufficiently
small $\epsilon>0$, then $P_\sigma(D)=D$;
\end{lem}
\begin{proof}
(1)
This follows from Nakayama [\ref{Nakayama}, Theorem 3.5.3].

(2) This is similar to Nakayama [\ref{Nakayama}, Lemma 3.5.1]. We have
$$
P_\sigma(f^*D+E)+N_\sigma(f^*D+E)=f^*D+E\equiv E/X
$$
hence $N_\sigma(f^*D+E)-E$ is anti-nef on the very general curves$/X$ of each component of $E$, and its
pushdown on $X$ is effective. Then, by Shokurov's generalized negativity lemma (cf. [\ref{B-lc-flips}, Lemma 3.3]),
we have $N_\sigma(f^*D+E)-E\ge 0$. Therefore,
$$
P_\sigma(f^*D+E)=P_\sigma(f^*D)
$$
by Nakayama [\ref{Nakayama}, Lemma 2.1.5].

(3)
Let $P=P_\sigma(D)$ and $N=N_\sigma(D)$.
Choose $a>0$ so that $M\le aN$. Then, 
$$
N_\sigma(D+M)=N_\sigma(P+N+M)\le N_\sigma(N+M)\le N+M
$$ 
Thus,  $P_\sigma(D+M)=P_\sigma(P+N+M)=P+C$ for some $0\le C\le N+M$.
 Then,
$$
(1+a)P=P_\sigma((1+a)D)=P_\sigma(D+aP+aN)
$$
$$
=P_\sigma(D+M+aP+aN-M)\ge P+C+aP
$$
which is possible only if $C=0$.

(4)
If $M^\sim$ is the birational transform of $M$ on $W$, then $M^\sim \le N_\sigma(f^*D)$.
By Nakayama [\ref{Nakayama}, Lemma 2.1.5],
$P_\sigma(f^*D-M^\sim)= P_\sigma(f^*D)$.
Now if we let $f^*M=M^\sim+E$, then
$$
P_\sigma(f^*D-M^\sim)=P_\sigma(f^*D-f^*M+E)=P_\sigma(f^*D-f^*M)
$$
by (2) hence $P_\sigma(D)=P_\sigma(D-M)$.

(5)
We may assume that $f\colon W\to X$ and $g\colon W\to Y$ give a common resolution. Then, $f^*D=g^*D_Y+E$
where $D_Y:=\phi_*D$ and $E\ge 0$ is exceptional$/Y$. By (2) we have
$$
P_\sigma(f^*D)=P_\sigma(g^*D_Y+E)=P_\sigma(g^*D_Y)
$$
hence
$$
\phi_*P_\sigma(D)=\phi_*f_*P_\sigma(f^*D)=\phi_*f_*P_\sigma(g^*D_Y)=g_*P_\sigma(g^*D_Y)=P_\sigma(D_Y)
$$

(6)
Take a  strictly decreasing sequence of numbers $\epsilon_i>0$ with $\lim \epsilon_i=0$.
Pick an ample divisor $G$ on $W$ so that $G+f^*A$ is ample.
Since $D+\epsilon_i A$ is a movable $\R$-divisor, $N_\sigma(f^*D+\epsilon_i (f^*A+G))$ is exceptional$/X$.
By going to the limit, we deduce that $N_\sigma(f^*D)$ is also exceptional$/X$ hence
$P_\sigma(D)=f_*P_\sigma(f^*D)=D$.\\
\end{proof}

 We will say that a pseudo-effective $\R$-Cartier divisor $D$ on a normal projective variety $X$
 \emph{birationally} has a  Nakayama-Zariski decomposition with nef positive part
if there is a resolution $f\colon W\to X$ so that $P_\sigma(f^*D)$ is nef. If $g\colon V\to W$ is any 
birational morphism from a smooth projective $V$, then $P_\sigma(g^*f^*D)=g^*P_\sigma(f^*D)$: 
indeed, since $P_\sigma(f^*D)$ is nef, $P_\sigma(g^*f^*D)\ge g^*P_\sigma(f^*D)$; on the other hand, 
in general $P_\sigma(g^*f^*D)\le g^*P_\sigma(f^*D)$ which follows from the 
generalized negativity lemma (cf. [\ref{B-lc-flips}, Lemma 3.3]) because 
$E:=g^*P_\sigma(f^*D)-P_\sigma(g^*f^*D)$
is exceptional$/W$ and it is anti-nef on the very general curves$/W$ of each component of $E$.

\begin{proof}(of Theorem \ref{t-main-1})
By assumptions,
there is a log resolution $f\colon W\to X$ on which we have $f^*(K_X+B)=P+N$ where
$P:=P_\sigma(f^*(K_X+B))$ is nef. We can write
$$
K_W+B_W=f^*(K_X+B)+E
$$
where $B_W,E\ge 0$, and $E$ is exceptional$/X$ whose support contains each prime exceptional$/X$ divisor
$D$ on $W$ if $a(D, X, B)>0$. The pair $(W,B_W)$ is dlt and a log minimal model of $(W,B_W)$ 
is also a log minimal model of $(X,B)$ by [\ref{B-lc-flips}, Remark 2.8]. By Lemma \ref{l-NZD} (2),
$$
P_\sigma(K_W+B_W)=P_\sigma(f^*(K_X+B)+E)=P_\sigma(f^*(K_X+B))=P
$$
So, we can replace $(X,B)$ with $(W,B_W)$ and we can assume that $K_X+B=P+N$
where $P=P_\sigma(K_X+B)$ is nef. Moreover, by taking a higher resolution if necessary we can also assume
 that $(X,B+N)$ is log smooth.

First assume that $N$ is a $\Q$-divisor. Let $\alpha$ be a sufficiently large number.
Now run an LMMP on $K_X+B+\alpha P$ with scaling of some ample divisor $A$.
By Theorem \ref{t-p-trivial-LMMP}, $P$ is numerically
trivial on each extremal ray in the process hence the nef
property of $P$ is preserved and the LMMP is also an LMMP on $K_X+B$.
After finitely many steps, the LMMP consists of only log flips hence we
reach a model $Y$ on which $K_Y+B_Y+\alpha P_Y+\epsilon A_Y$ is a movable $\R$-divisor
for each $\epsilon>0$.

Since
$$
N_\sigma(K_X+B+\alpha P+\alpha N)=N_\sigma((1+\alpha)(P+N))=(1+\alpha)N
$$
from Lemma \ref{l-NZD} (4) we get
$$
P_\sigma(K_X+B+\alpha P)=P_\sigma(P+N+\alpha P)=(1+\alpha)P
$$
By Lemma \ref{l-NZD} (5),
$$
P_\sigma(K_Y+B_Y+\alpha P_Y)=\phi_*P_\sigma(K_X+B+\alpha P)=(1+\alpha)P_Y
$$
where $\phi$ is the birational map $X\bir Y$. On the other hand, by Lemma \ref{l-NZD} (6),
$$
P_\sigma(K_Y+B_Y+\alpha P_Y)=K_Y+B_Y+\alpha P_Y=N_Y+(1+\alpha)P_Y
$$
Therefore, $N_Y=0$ and $K_Y+B_Y=P_Y$ which is nef. This means that the LMMP terminates and
we get a log minimal model of $(X,B)$.

Now we treat the general case when $N$ is not necessarily a $\Q$-divisor. Since 
$(X,B+N)$ is log smooth we can find an $\R$-divisor $G$
with sufficiently small coefficients such that $\Supp G\subseteq \Supp N$,
$(X,B'=B+G)$ is dlt, $N'=N+G\ge 0$ is a $\Q$-divisor, and that $\Supp N'=\Supp N$. Note that
$G$ may not be effective but this does not cause any problem.

We get the
decomposition $K_X+B'=P+N'$ where by Lemma \ref{l-NZD} (3)(4) we have $P=P_\sigma(K_X+B')$ which
is nef. By the arguments above we can run a $P$-trivial
LMMP on $K_X+B'$ which ends up with a log minimal model $(Y,B_Y')$. The LMMP contracts $N'$
hence it contracts $N$ as well. Let $f\colon W\to X$ and $g\colon W\to Y$ be a common resolution.
Put $R=f^*(K_X+B)-g^*(K_Y+B_Y)$. Since the LMMP is $P$-trivial, $f_*g^*(P_Y)=P$ from which we get
$$
f_*R=K_X+B-f_*g^*(K_Y+B_Y)=K_X+B-f_*g^*(P_Y)=K_X+B-P=N\ge 0
$$
Thus, $R\ge 0$ by the negativity lemma. Moreover, any prime divisor $D$ on $X$ contracted by $X\bir Y$
is a component of $f_*R=N$ hence 
$$
a(D,X,B)<a(D,Y,B_Y)
$$
Therefore, $(Y,B_Y)$ is a log minimal model of $(X,B)$.
\end{proof}

\vspace{0.3cm}
\section{Fujita and CKM Zariski decompositions}\label{s-FCKMZD}

Let $D$ be an $\R$-Cartier divisor on a normal projective variety $X$. A \emph{Fujita-Zariski decomposition}
for $D$ is an expression $D=P+N$ such that

$\bullet$ $P$ is nef, $N\ge 0$, and

$\bullet$ if $f\colon W\to X$ is a birational morphism from a normal projective variety and $f^*D=P'+N'$
with $P'$ nef and $N'\ge 0$, then $P'\le f^*P$.

On the other hand, a \emph{Cutkosky-Kawamata-Moriwaki-Zariski}
(CKM-Zariski for short) decomposition for $D$ is an expression $D=P+N$ such that

$\bullet$ $P$ is nef, $N\ge 0$, and

$\bullet$ the maps $H^0(X,\rddown{mP})\to H^0(X,\rddown{mD})$ are isomorphisms for all
$m\in \N$.

In both decompositions, $P$ is called the positive part.\\

\begin{lem}\label{l-FZD}
Let $D$ be an $\R$-Cartier divisor on a normal projective variety $X$ with a Fujita-Zariski decomposition
$D=P+N$. Let $G$ be an $\R$-Cartier divisor such that $\Supp G\subseteq \Supp N$ and
 $N+G\ge 0$. Then, $P+N+G$ is a Fujita-Zariski decomposition of $D+G$ with $P$ the positive part.
\end{lem}
\begin{proof}
Put $M=N+G$. Let $f\colon W\to X$ be a birational morphism from a normal projective variety and
assume $f^*(D+G)=P'+M'$
where $P'$ is nef and $M'\ge 0$. We need to show that $f^*P\ge P'$.
There is $a>0$ such that $aN=M+L$ with $L\ge 0$. Then, we have
\begin{equation*}
\begin{split}
 f^*(aD)=af^*P+af^*N & =af^*P+f^*M+f^*L\\
 & =(a-1)f^*P+f^*(D+G)+f^*L\\
 & =(a-1)f^*P+P'+M'+f^*L
\end{split}
\end{equation*}

Since $aD=aP+aN$ is a Fujita-Zariski decomposition with $aP$ the positive part, 
we deduce that $af^*P\ge (a-1)f^*P+P'$ hence $f^*P\ge P'$.
\end{proof}

Let $(X,B)$ be a projective klt pair. Kawamata [\ref{Kawamata-ZD}] showed that if $K_X+B$ is a big
$\Q$-divisor and if it has a CKM-Zariski decomposition, then the log canonical
ring of $(X,B)$ is finitely generated, i.e. $(X,B)$ has a log canonical model (see also Moriwaki [\ref{Moriwaki}] 
and Fujita [\ref{Fujita}]).
We prove a similar result for lc pairs where we replace log canonical model by log minimal model.

\begin{thm}
Let $(X,B)$ be a projective lc pair such that $K_X+B$ is big. Then, the following are
equivalent:

(1) $K_X+B$ birationally  has a Fujita-Zariski decomposition;

(2) $K_X+B$ birationally  has a CKM-Zariski decomposition;

(3) $(X,B)$ has a log minimal model;
\end{thm}
\begin{proof}
For the implication (3)$\implies$(2): see the proof of [\ref{B-WZD}, Theorem 1.5].
For (2)$\implies$(1): see Prokhorov [\ref{Prokhorov}, Section 7].

It remains to show that (1)$\implies$(3). By assumptions,
there is a log resolution $f\colon W\to X$ 
such that we have a Fujita-Zariski decomposition
$f^*(K_X+B)=P+N$. 
We can write
$$
K_W+B_W=f^*(K_X+B)+E
$$
where $B_W,E\ge 0$, and $E$ is exceptional$/X$ whose support contains each prime exceptional$/X$ divisor
$D$ on $W$ if $a(D, X, B)>0$. The pair $(W,B_W)$ is dlt and a log minimal model of $(W,B_W)$ 
is also a log minimal model of $(X,B)$ by [\ref{B-lc-flips}, Remark 2.8].
Moreover, $K_W+B_W=P+N+E$ is a Fujita-Zariski decomposition
with $P$ being the positive part: indeed, assume that $g\colon V\to W$ is a
birational morphism from a normal projective variety and that $g^*(K_W+B_W)=P'+N'$ where
$P'$ is nef and $N'\ge 0$. Then, $g^*f^*(K_X+B)+g^*E=P'+N'$ hence
$N'-g^*E$ is anti-nef$/X$ and the negativity lemma shows that $N'-g^*E\ge 0$. So,
$g^*P+g^*N=P'+N'-g^*E$ and from this we obtain $g^*P\ge P'$ which proves the claim.

By replacing $(X,B)$ with $(W,B_W)$ we could assume
that $K_X+B=P+N$ is a Fujita-Zariski decomposition and that $(X,B+N)$ is log smooth.
Since $K_X+B$ is big, we can write $K_X+B\sim_\R A+L$ where $A$ is an ample $\R$-divisor and
$L\ge 0$.

Assume that $N$ is a $\Q$-divisor. Then, by Theorem \ref{t-p-trivial-LMMP},
we can run a $P$-trivial LMMP on $K_X+B+\alpha P$ with scaling of some multiple of
$A$, for some sufficiently large
number $\alpha$. We arrive at a model $Y$ on which $K_Y+B_Y+\alpha P_Y+\epsilon A_Y$ is
semi-ample for some sufficiently small $\epsilon>0$.
Then, we have
\begin{equation*}
\begin{split}
(1+\alpha+\epsilon)(P_Y+N_Y) & = (1+\alpha+\epsilon)(K_Y+B_Y)\\
& \sim_\R K_Y+B_Y+\alpha P_Y+\epsilon A_Y+\alpha N_Y+\epsilon L_Y
\end{split}
\end{equation*}
Then, for some numerically trivial $\R$-divisor $Q$ we have
$$
(1+\alpha+\epsilon)(P_Y+N_Y)=Q+K_Y+B_Y+\alpha P_Y+\epsilon A_Y+\alpha N_Y+\epsilon L_Y
$$
Since $X\bir Y$ was obtained as a $P$-trivial LMMP$/Z$, $K_Y+B_Y=P_Y+N_Y$ is still a Fujita-Zariski decomposition 
by reasoning as in the second paragraph of this proof.
So, we get
$$
(1+\alpha+\epsilon)P_Y\ge Q+K_Y+B_Y+\alpha P_Y+\epsilon A_Y
$$
hence
$$
(1+\alpha+\epsilon)N_Y\le \alpha N_Y+\epsilon L_Y
$$
By our choice of $\epsilon$, the latter inequality is possible only if $N_Y=0$ in which case
$K_Y+B_Y=P_Y$ is nef and we get the desired log minimal model.

Now we deal with the case when $N$ may not be a $\Q$-divisor.
We argue as in the proof of Theorem \ref{t-main-1}.
Pick an $\R$-divisor $G$
with sufficiently small coefficients such that $\Supp G\subseteq \Supp N$,
$(X,B'=B+G)$ is dlt, $N'=N+G\ge 0$ is a $\Q$-divisor, and that $\Supp N'=\Supp N$. We get the
decomposition $K_X+B'=P+N'$ which is a Fujita-Zariski decomposition by Lemma \ref{l-FZD} with $P$ being
the positive part. By the arguments above we can run a $P$-trivial
LMMP on $K_X+B'$ which ends up with a log minimal model $(Y,B_Y')$. The LMMP contracts $N'$
hence it contracts $N$ as well. Let $f\colon W\to X$ and $g\colon W\to Y$ be a common resolution.
Put $R=f^*(K_X+B)-g^*(K_Y+B_Y)$. Since the LMMP is $P$-trivial, $f_*g^*(P_Y)=P$ from which we get
$$
f_*R=K_X+B-f_*g^*(K_Y+B_Y)=K_X+B-f_*g^*(P_Y)=K_X+B-P=N\ge 0
$$
Thus, $R\ge 0$ by the negativity lemma.
Moreover, any prime divisor $D$ on $X$ contracted by $X\bir Y$
is a component of $f_*R=N$ hence 
$$
a(D,X,B)<a(D,Y,B_Y)
$$
Therefore, $(Y,B_Y)$ is a log minimal model of $(X,B)$.
\end{proof}

We can generalize the theorem as in the following result.  
Recall that an $\R$-Cartier divisor $D$
is said to be \emph{abundant} if $\kappa(D)=\kappa_\sigma(D)$ (cf. [\ref{Lehmann}]) where 
$\kappa_\sigma(D)$ is the numerical Kodaira dimension defined by Nakayama [\ref{Nakayama}]. 
In particular, any big $\R$-Cartier divisor is abundant.

\begin{thm}
Let $(X,B)$ be a projective lc pair such that $K_X+B$ birationally has a Fujita-
Zariski decomposition. Assume that $K_X+B$
is abundant. Then $(X,B)$ has a log minimal model.
\end{thm}
\begin{proof}
We will show that the Fujita-Zariski decomposition coincides with the Nakayama-Zariski
decomposition so we can apply Theorem \ref{t-main-1}.
By taking a log resolution we may assume $(X,B)$ dlt and that $K_X+B=P+N$ is the Fujita-Zariski
decomposition. In particular, this is also a CKM-Zariski decomposition hence $F\ge N$ where $F$ is 
the asymptotic fixed part $F:= \lim \frac{1}{m}\Fix|\rddown{m(K_X + B)}|$.
On the other hand, since $K_X+B$ is abundant, $F=N_\sigma(K_X+B)$ by [\ref{Lehmann}, Proposition 6.4]. 
We immediately obtain the conclusion from $F\ge N \ge N_\sigma(K_X+B)=F$.
\end{proof}

Using Theorem \ref{t-main-2} we can take care of Theorem \ref{t-main-4}.

\begin{proof}(of Theorem \ref{t-main-4})
If $(X,B)$ is $\Q$-factorial dlt, the proof is actually trivial since we can 
easily get rid of $P$ and the lc centres hence the claim reduces to the well-known 
finite generation for klt pairs. 
But presence of lc singularities causes difficulties.
Let $D:=K_X+B+P$ and let $A$ be an ample $\Q$-divisor. Perhaps after replacing 
$A$ with some small multiple, by assumptions, we can write $D\sim_\Q A+L$ where $L\ge 0$ and $\Supp L$ 
does not contain any lc  centre of $(X,B)$. Thus, we can find a boundary $\Delta$ such that 
$$
K_X+\Delta\sim_\Q K_X+B+P+\epsilon A+\epsilon L\sim_\Q (1+\epsilon) D
$$
and $(X,\Delta)$ is lc where $\epsilon$ is a small rational number. Moreover, we may assume 
that the lc centres of $(X,\Delta)$ are exactly the lc centres of $(X,B)$. By replacing $(X,B+P)$ 
with $(X,\Delta)$ from now on we can assume that $P=0$. 

Since $K_X+B$ is big and it birationally has a CKM-Zariski decomposition, $(X,B)$ has a log minimal model 
by Theorem \ref{t-main-2}. So, we can run an LMMP on $K_X+B$ with scaling of some ample divisor 
which ends up with a weak lc model of $(X,B)$ as defined in Section 2: all the necessary ingredients for such an LMMP 
are already established; more precisely we need the cone and contraction theorem for lc pairs 
[\ref{Fujino-MMP}, subsection 3.3.3], 
existence of lc flips [\ref{B-lc-flips}, Corollary 1.2], 
certain properties of extremal rays [\ref{Fujino-MMP}, Proposition 3.23], and termination with scaling 
[\ref{B-lc-flips}, Theorem 1.9]. 

We will show that the LMMP preserves the condition 
that ${\bf{B}}_+(D)$ does not contain any lc centre of $(X,B)$. Indeed let $X\bir X^+/Z$ be a $K_X+B$-flip. 
For a divisor $F$ on $X$, $F^+$ denotes the birational transform on $X^+$. Similarly for a 
divisor $C^+$ on $X^+$, $C$ will denote its birational transform on $X$. Since $D^+=K_{X^+}+B^+$ is ample$/Z$, 
there is a divisor $H^+$ on $X^+$ which is the pullback of a sufficiently ample divisor on $Z$ 
so that $C^+:=D^++H^+$ is ample and effective. Since $C=D+H$ and $D,H$ are both $\Q$-Cartier, $C$ is also $\Q$-Cartier. So, we can write 
$A=\epsilon C+F$ where $A\ge 0$ and $F\ge 0$ are ample divisors and $\epsilon>0$ is rational. Moreover, we can assume that   
$D\sim_\Q A+L$ where $L\ge 0$ and $\Supp L$ 
does not contain any lc  centre of $(X,B)$. Choose $C^+$ and $F$ so that $\Supp C^+$ does not  
contain any lc  centre of $(X^+,B^+)$ and that $\Supp F$ does not contain any lc  centre of $(X,B)$. 
Then, $D^+\sim_\Q \epsilon C^++F^++L^+$ where $\Supp (\epsilon C^++F^++L^+)$ does not contain any
lc  centre of $(X^+,B^+)$. This proves that ${\bf{B}}_+(D^+)$ does not contain any 
lc  centre of $(X^+,B^+)$. A similar reasoning can be applied to the case of a divisorial contraction. 

By replacing $X$ with the weak lc model, 
we can assume that $K_X+B$ is nef and that $(K_X+B)|_V$ is big for any lc centre $V$. 
By Fujino-Gongyo [\ref{Fujino-Gongyo}, Theorem 4.2], $K_X+B$ is semi-ample which implies that 
$R(K_X+B)$ is a finitely generated $k$-algebra.  
\end{proof}

\vspace{0.3cm}
\section{Weak Zariski decompositions}\label{s-WZD}

Recall that an $\R$-Cartier divisor $D$ on a normal projective variety $Y$ birationally 
has a weak Zariski decomposition if there is a resolution $f\colon W\to Y$ such that $f^*D=P+N$
where $P$ is nef and $N\ge 0$. Assume that for any projective lc pair $(X,B)$ of dimension 
$\le d$ with $K_X+B$ pseudo-effective birationally we have
a weak Zariski decomposition for $K_X+B$. Can one construct log minimal models for such pairs?
In this section, we outline a strategy to tackle this problem.

Let $(X,B)$ be a projective lc pair of dimension $\le d$ with $K_X+B$ having a weak Zariski decomposition. By taking a log
resolution we may assume that the pair is log smooth and that $K_X+B=P+N$ is
the decomposition where $P$ is nef and $N\ge 0$.

\emph{Step 1.}
By the proof of
[\ref{B-WZD}, Theorem 1.5], we may assume that $\Supp N\subseteq \Supp \rddown{B}$.

\emph{Step 2.}
Let $\alpha$ be a sufficiently large number and run an LMMP
on $K_X+B+\alpha P$ with scaling of some ample divisor. Then, by Theorem \ref{t-p-trivial-LMMP}, $P$ is trivial on each
extremal ray contracted in the process and we get an LMMP on $K_X+B$.
By replacing $X$ we may assume that the LMMP consists of only flips.
We should show that this LMMP terminates by using special termination arguments.
If $S$ is a component of $\rddown{B}$ and if we put $K_S+B_S=(K_X+B)|_S$ and $P_S=P|_S$,
then we need to show that  the induced LMMP on $K_S+B_S+\alpha P_S$ terminates.
This is obviously related to the material in Section 3. This should be somehow
derived from existence of weak Zariski decompositions and log minimal models in dimension $<d$.

\emph{Step 3.} If Step 2 is done successfully, then we can assume that $K_X+B+\alpha P$ is nef.
Next, run an LMMP on $K_X+B$ with scaling of $\alpha P$ as in the proof of Theorem \ref{t-l-pairs-termination}.
Again as in Step 2, we need to use special termination
arguments to show that the LMMP terminates. In other words, for a component $S$
of $\rddown{B}$ we have $K_S+B_S+\alpha P_S$ nef and we need to show that the induced
LMMP on $K_S+B_S$ with scaling of $\alpha P_S$ terminates. As mentioned in Step 2, this is
related to Section 3.

In order to prove the terminations required in Steps 2 and 3 we probably need to
generalize  [\ref{B-II}, Theorem 1.5] to the setting of polarized pairs. There are indications that 
in the strategy above it might be better to start with a polarized pair rather than a usual pair $(X,B)$.


\vspace{2cm}

\flushleft{DPMMS}, Centre for Mathematical Sciences,\\
Cambridge University,\\
Wilberforce Road,\\
Cambridge, CB3 0WB,\\
UK\\
email: c.birkar@dpmms.cam.ac.uk\\
email: zh262@dpmms.cam.ac.uk

\vspace{1cm}

\end{document}